\documentclass[12pt,reqno]{amsart}
\usepackage{amsfonts}
\usepackage{amssymb,color}

\usepackage[colorlinks=true]{hyperref}

\hypersetup{urlcolor=blue, citecolor=blue}

\usepackage[margin=3cm, a4paper]{geometry}

\def\normo#1{\left\|#1\right\|}
\def\normb#1{\big\|#1\big\|}
\def\abs#1{|#1|}
\def\aabs#1{\left|#1\right|}

\def\norm#1{\|#1\|}
\def\jb#1{\langle#1\rangle}
\def\wt#1{\widetilde{#1}}

\newcommand{\N}{{\mathbb N}}

\newcommand{\E}{{\mathbb E}}
\newcommand{\R}{{\mathbb R}}
\newcommand{\C}{{\mathbb C}}
\newcommand{\Z}{{\mathbb Z}}
\newcommand{\ft}{{\mathcal{F}}}
\newcommand{\Hl}{{\mathcal{H}}}

\newcommand{\les}{{\;\lesssim\;}}

\newcommand{\Lr}{{\mathcal{L}}}

\def\jb#1{\langle#1\rangle}
\def\norm#1{\|#1\|}
\def\normo#1{\left\|#1\right\|}
\def\normb#1{\big\|#1\big\|}
\def\wt#1{\widetilde{#1}}

\def\aabs#1{\left|#1\right|}

\newcommand{\F}{\mathcal{F}}

\newcommand{\cS}{\mathcal{S}}

\newcommand{\cir}{\mathbb{S}}

\newcommand{\al}{\alpha}
\newcommand{\be}{\beta}
\newcommand{\ga}{\gamma}

\newcommand{\e}{\varepsilon}
\newcommand{\fy}{\varphi}
\newcommand{\om}{\omega}

\newcommand{\te}{\theta}
\newcommand{\s}{\sigma}

\newcommand{\x}{\xi}

\newcommand{\Des}{\Delta_\omega}

\newcommand{\lec}{\lesssim}

\newcommand{\I}{\infty}

\newcommand{\EQ}[1]{\begin{align*}\begin{split} #1 \end{split}\end{align*}}
\newcommand{\EQn}[1]{\begin{align}\begin{split} #1 \end{split}\end{align}}
\newcommand{\Del}[1]{}

\newcommand{\pt}{&}
\newcommand{\pr}{\\ &}
\newcommand{\pq}{\quad}

\numberwithin{equation}{section}

\newtheorem{thm}{Theorem}[section]

\newtheorem{lem}[thm]{Lemma}
\newtheorem{prop}[thm]{Proposition}

\theoremstyle{remark}
\newtheorem{rem}[thm]{Remark}

\begin{document}

\title[Boundary Strichartz estimates]{On the boundary Strichartz estimates for wave and Schr\"odinger equations}

\author[Z. Guo, J. Li, K. Nakanishi and L. Yan]{Zihua Guo, Ji Li, Kenji Nakanishi and Lixin Yan} 
\address{School of Mathematical Sciences, Monash University, VIC 3800, Australia}
\email{zihua.guo@monash.edu}

\address{Department of Mathematics, Macquarie University, NSW, 2109, Australia}
\email{ji.li@mq.edu.au}

\address{Research Institute for Mathematical Sciences Kyoto University Kyoto 606-8502, JAPAN}
\email{kenji@kurims.kyoto-u.ac.jp}

\address{Department of Mathematics, Sun Yat-sen (Zhongshan) University, Guangzhou, 510275, P.R. China}
\email{mcsylx@mail.sysu.edu.cn}

\thanks{Z. G. and J. L. are supported by ARC DP170101060.}

\begin{abstract}
We consider the $L_t^2L_x^r$ estimates for the solutions to the wave and Schr\"odinger equations in high dimensions. For the homogeneous estimates, we show $L_t^2L_x^\infty$ estimates fail at the critical regularity in high dimensions by using stable L\'evy process in $\R^d$. Moreover, we show that some spherically averaged $L_t^2L_x^\infty$ estimate holds at the critical regularity. As a by-product we obtain Strichartz estimates with angular smoothing effect. For the inhomogeneous estimates, we prove double $L_t^2$-type estimates.
\end{abstract}

\maketitle

\tableofcontents

\section{Introduction}
In this paper, we study the space-time (Strichartz) estimates for solutions to the
Schr\"odinger type dispersive equations
\begin{align}\label{eq:schr}
i\partial_{t}u+D^a u=g,\quad u(0,x)=f(x)
\end{align}
where $u(t,x):\R\times \R^{d}\to \C$ is the unknown function, $D=\sqrt{-\Delta}$, $0<a\leq 2$. 
Two typical examples of \eqref{eq:schr} are of particular interest: the wave
equation ($a=1$) and the Schr\"odinger equation ($a=2$). The rest cases ($0<a<2$) are also known as fractional Schr\"odinger equation which has attracted many attentions recently and is a fundamental equation of fractional quantum mechanics, which was derived by Laskin (see \cite{Laskin1,Laskin2}) as a result of extending the Feynman path integral, from the Brownian-like to L\'{e}vy-like quantum mechanical paths.

The Strichartz estimates address the following space-time estimates for the solution $u$ to \eqref{eq:schr}, e.g. when $g=0$, 
\begin{align}\label{eq:striest}
\norm{e^{itD^a}P_0f}_{L_t^qL_x^p(\R\times \R^d)}\les \norm{f}_{L^2},
\end{align}
where $e^{itD^a}$, $P_0$ and $L_t^qL_x^p$ are defined in the end of this section. 
In a pioneering paper Strichartz \cite{Str} first proved \eqref{eq:striest} for the case
$q=p$ by the Fourier restriction method  and then the estimates were substantially extended by many
authors. It is now well-known (see \cite{KT} and references therein) that for $d\geq 1$
the estimate \eqref{eq:striest} holds if and only if $(q,p)$ satisfies the {\it admissible conditions}:
\begin{align}\label{eq:adm}
2\leq q,p\leq \infty,\quad\frac{1}{q}\leq
\frac{d-d_a+2}{2}(\frac{1}{2}-\frac{1}{p}),\quad (q,p,d)\ne (2,\infty,d_a),
\end{align}
where
\begin{align}\label{eq:da}
d_a=
\begin{cases}
2, \quad a\ne 1;\\
3, \quad a=1.
\end{cases}
\end{align}
In particular, the endpoint estimates $(q,p)=(2,2+\frac{4}{d-d_a})$ for $d>d_a$ were proved in \cite{KT}, and the failure of $(q,p,d)=(2,\infty,d_a)$ was proved in \cite{MSmith} for $a=1,2$.

By the scaling invariance of the equation \eqref{eq:schr}: for $\lambda>0$
\begin{align*}
f(x)\to f(\lambda x), \quad u\to u(\lambda^a t, \lambda x),
\end{align*}
the Minkowski inequality and the Littlewood-Paley square function theorem, one can get from \eqref{eq:striest} the following frequency-global Strichartz estimates
\begin{align}\label{eq:striestgl}
\norm{e^{itD^a}f}_{L_t^qL_x^p}\les \norm{f}_{\dot H^{s}},
\end{align}
if $(q,p)$ satisfies the admissible conditions \eqref{eq:adm}, $p\ne \infty$ and the natural scaling condition
\begin{align*}
s=(\frac{1}{2}-\frac{1}{p})d-\frac{a}{q}.
\end{align*}
On the other hand, for $p=\infty$, the estimates \eqref{eq:striestgl}, namely
\begin{align}\label{eq:striestgl2}
\norm{e^{itD^a}f}_{L_t^qL_x^\infty}\les \norm{f}_{\dot H^{\frac{d}{2}-\frac{a}{q}}}
\end{align}
need special treatment due to the failure of the Littlewood-Paley theory in $L^\infty$. For $q>2$ and $d\geq d_a$ (or $q\geq 4$, $d=1$ and $a\neq 1$), one can prove \eqref{eq:striestgl2} by interpolations or directly by $TT^*$ method  (see Section 2). For wave equation $a=1$, \eqref{eq:striestgl2} was studied in \cite{FangWang} and in particular the estimate \eqref{eq:striestgl2} was shown to be false for $(a, q,d)=(1,4, 2)$. Therefore, the only unknown estimates for \eqref{eq:striestgl} are the endpoints $(q,p)=(2,\infty)$ for $d>d_a$ which can not be handled by interpolations or $TT^*$ method as $(2,\infty)$ lies on the boundary of admissible conditions. When $d=d_a$, these endpoint estimates (even weaker version \eqref{eq:striest}) fail and it was known that the failure is logarithmic due to the $t$-integration on the whole line. Indeed, the following estimate was proved by Tao in \cite{Tao2}:
\begin{align*}
\norm{e^{it\Delta}P_0f}_{L_t^2L_x^\infty(I\times \R^2)}\les \log(2+|I|)^{1/2}\norm{f}_{L^2(\R^2)}.
\end{align*}
For $d>d_a$, \eqref{eq:striest} holds for $(q,p)=(2,\infty)$ from which we can get the following two estimates
\begin{align*}
\norm{e^{itD^a}f}_{L_t^2\text{BMO}_x}\les \norm{f}_{\dot H^{s_c}},\quad
\norm{e^{itD^a}f}_{L_t^2L^\infty_x}\les \norm{f}_{H^s}
\end{align*}
where $s>s_c:=\frac{d-a}{2}$, and BMO is the space of bounded mean oscillation. BMO is usually a good substitute for $L^\infty$ in harmonic analysis. Thus we see the $L_t^2L_x^\infty$ estimate is logarithmically missing at the critical regularity. 

The purpose of this paper is to study various $L_t^2L_x^r$-type estimates for \eqref{eq:schr}.  Our first result is

\begin{thm}\label{thm1}
If $0<a\leq 2, d>d_a$ with $d_a$ given by \eqref{eq:da}, then the following estimate fails:
\begin{align}\label{eq:crit}
\norm{e^{itD^a}f}_{L_t^2L^\infty_x}\les \norm{f}_{H^{\frac{d-a}{2}}}.
\end{align}
\end{thm}

Different from the case $d=d_a$, the logarithmic failure of the above estimate is due to the summation over the frequency. We will use the $a$-stable L\'evy process to prove the above theorem. When $a=2$, this process reduces to the Brownian motion. Our ideas are inspired by \cite{MSmith}. 

In spite of the results in \cite{MSmith}, Tao (see \cite{Tao}) showed in the radial case the $L_t^2L_x^\infty$ estimate for the 2D Schr\"odinger equation holds. Actually, he proved a spherically averaged estimate:
\begin{align*}
\norm{e^{it\Delta}f}_{L_t^2\Lr^\infty_\rho L^2_{\om}(\R\times \R^2)}\les \norm{f}_{L^2_x(\R^2)},
\end{align*}
where $L_t^2\Lr^\infty_\rho L^2_{\om}$ is defined by \eqref{eq:Lpq}. Moreover,  he proved there is an $\epsilon$-angular smoothing effect: $\exists \epsilon>0$ such that
\begin{align*}
\norm{\Lambda_\om^\epsilon e^{it\Delta}f}_{L_t^2\Lr^\infty_\rho L^2_{\om}(\R\times \R^2)}\les \norm{f}_{L^2},
\end{align*}
where $\Lambda_\om$ is the angular derivative (see the end of this section).
In Theorem 5.1 of \cite{MNNO}, an upper bound on the smoothing effect $\epsilon \leq 1/3$ was shown. 
Our second result extends Tao's result to the cases $0<a<\infty$, $d\geq d_a$.

\begin{thm}\label{thm2}

(1) If $0<a<\infty$, $d\geq 3$ and $s<\frac{d-2}{2}$, then
\begin{align*}
\norm{\Lambda_\om^{s}e^{itD^a}f}_{L_t^2\Lr^\infty_\rho L^2_{\om}}\les \norm{f}_{\dot H^{\frac{d-a}{2}}}.
\end{align*}

(2) If $1<a<\infty$, $d\geq 2$ and $s<\frac{1}{7}+\frac{d-2}{2}$, then
\begin{align*}
\norm{\Lambda_\om^{s} e^{itD^a}f}_{L_t^2\Lr^\infty_\rho L^2_{\om}}\les \norm{f}_{\dot H^{\frac{d-a}{2}}}.
\end{align*}
\end{thm}

Besides its own interest, the $L_t^pL_x^\infty$ estimates play important roles in the study of the nonlinear problems. Especially, it is useful for fractional Schr\"odinger equations ($a<2$) when the classical Strichartz estimates have a loss of derivatives, e.g. see \cite{HSire}. In the appendix we apply the above theorem to study the cubic fractional Schr\"odinger equations. 

The general spherically averaged estimates
\begin{align}\label{eq:striang}
\norm{e^{itD^a}P_0f}_{L_t^q\Lr_\rho^pL_\om^2}\les \norm{f}_{L^2}
\end{align}
were also studied. It was known that \eqref{eq:striang} allows a wider range of indices $(q,p)$ than \eqref{eq:striest}. 
For the wave equation $a=1$ and $d\geq 2$,  the optimal range of $(q,p)$ for \eqref{eq:striang} is  (see \cite{JWY}  and references therein, \cite{SSW} for $d=2$): $(q,p)=(\infty,2)$ or
\begin{align}\label{eq:radm}
2\leq q,p\leq \infty,\quad\frac{1}{q}<
({d-\frac{d_a}{2}+\frac{1}{2}})(\frac{1}{2}-\frac{1}{p}).
\end{align}
For the case $a>1$ and $d\geq 2$, \eqref{eq:striang} holds if $(q,p)$ satisfies either \eqref{eq:radm} or
\begin{align*}
2\leq q,p\leq \infty,\quad\frac{1}{q}=
({d-\frac{d_a}{2}+\frac{1}{2}})(\frac{1}{2}-\frac{1}{p}), \quad (q,p)\neq (2, \frac{4d-2}{2d-3}).
\end{align*}
These conditions are also necessary except the endpoints $(2, \frac{4d-2}{2d-3})$ which are still open (see \cite{Guo} and references therein).  To apply these estimates to the nonlinear problems, one needs inhomogeneous estimates	
\begin{align*}
\normo{\int_0^te^{i(t-s)D^a}P_0g(s)ds}_{L_t^q\Lr^p_\rho L^2_{\om}}\les \norm{g}_{L_t^{\tilde q'}\Lr_\rho^{r'}L_\om^2}
\end{align*}
which can be obtained by the standard Christ-Kiselev lemma. However, by Christ-Kiselev lemma one misses the double $L_t^2$ type estimates, namely $q=\tilde q=2$ in the above estimates. 
Our last result is concerned with the generalized double endpoint inhomogeneous Strichartz estimates.

\begin{thm}\label{thm3}
Let $0<a\leq 2$ and $d>d_a$. Assume $p,r>\frac{4d+2-2d_a}{2d-d_a-1}$.  Then the following estimate holds
\begin{align}\label{eq:inh}
\normo{\int_0^te^{i(t-s)D^a}P_0g(s)ds}_{L_t^2\Lr^p_\rho L^2_{\om}}\les \norm{g}_{L_t^2\Lr_\rho^{r'}L_\om^2}.
\end{align}
\end{thm}
\begin{rem}
The condition $d>d_a$ is necessary in our proof. We do not know whether \eqref{eq:inh} holds for $d=d_a$.
\end{rem}

\medskip

\noindent{\bf Notations.} 
We use $\ft(f)$ and $\widehat{f}$ to denote the Fourier
transform of $f$: $\hat f(\xi)=\int_{\R^d}e^{-ix\cdot\xi}f(x)dx$.
For $a>0$, define $S_a(t)=e^{itD^a}=\ft^{-1}e^{it|\xi|^a}\ft$. 

Let $\eta: \R\to [0, 1]$ be an even, smooth, non-negative and radially decreasing
function which is supported in $\{\xi:|\xi|\leq \frac{8}{5}\}$ and
$\eta\equiv 1$ for $|\xi|\leq \frac{5}{4}$. For $k\in \Z$, let
$\chi_k(\xi)=\eta(\frac{\xi}{2^k})-\eta(\frac{\xi}{2^{k-1}})$ and $\chi_{\leq
k}(\xi)=\eta(\frac{\xi}{2^k})$, and define Littlewood-Paley operators $P_k, P_{\leq k}$ on $L^2(\R^d)$ by
$\widehat{P_ku}(\xi)=\chi_k(|\xi|)\widehat{u}(\xi),\,\widehat{P_{\leq
k}u}(\xi)=\chi_{\leq k}(|\xi|)\widehat{u}(\xi)$.

$\Des$ denotes the Laplace-Beltrami operator
on the unite sphere $\cir^{d-1}$ endowed with the standard metric
$g$ and with the standard measure $d\omega$. Let $\Lambda_\omega=\sqrt{1-\Delta_\omega}$.
Denote $L_\omega^p=L_\omega^p(\cir^{d-1})=L^p(\cir^{d-1}:d\omega)$,
$\Hl_p^s=\Hl_p^s(\cir^{d-1})=\Lambda_\omega ^{-s}L_\omega^p$.

Let $L^p(\R^d)$ denote the usual Lebesgue space, and $\Lr^p(\R^+)=L^p(\R^+:r^{d-1}dr)$.
$\Lr_r^pL_\om^q $ are Banach spaces on $\R^d$ defined by the following norms: 
\begin{align}\label{eq:Lpq}
\norm{f}_{\Lr_{r}^pL_\om^q}=\big\|{\norm{f(r\om)}_{L_\om^q}}\big\|_{\Lr_{r}^p}
\end{align}
with $x=r\om,\ \om\in \cir^{d-1}$.
Let $X$ be a Banach space on $\R^d$. $L_t^qX$
denotes the space-time function space on $\R\times \R^d$ with the norm
$\norm{u}_{L_t^qX}=\big\|\norm{u(t,\cdot)}_X\big\|_{L_t^q}$.
$H^s_p$ ($\dot{H}_p^s$) are the usual inhomogeneous (homogeneous) Sobolev spaces on $\R^d$.

\section{$L_t^2L_x^\infty$ estimates fail at the critical regularity}

In this section, we consider the $L_t^qL_x^\infty$ estimates. First we prove the following proposition

\begin{prop}
Let $a>0$, $d\geq 1$ and $2< q<\infty$. Then
\begin{align}\label{eq:Strsec2}
\norm{e^{itD^a}f}_{L_t^qL_x^\infty}\les \norm{f}_{\dot H^{\frac{d}{2}-\frac{a}{q}}}
\end{align}
holds if assuming either of the following conditions:
\begin{itemize}
\item $a=1$, $\frac{2}{q}< \frac{d-1}{2}$.

\item $a\neq 1$, $\frac{2}{q}\leq \frac{d}{2}$.
\end{itemize}
\end{prop}
\begin{proof}
By $TT^*$ method, the estimate \eqref{eq:Strsec2} is equivalent to
\begin{align*}
\normo{\int e^{i(t-s)D^a}D^{\frac{2a}{q}-d}g(s)ds}_{L_t^qL_x^\infty}\les \norm{g}_{L_t^{q'}L_x^1}.
\end{align*}
By the dispersive estimates given in the lemma below and the Hardy-Littlewood-Sobolev inequality we get
\begin{align*}
\normo{\int e^{i(t-s)D^a}D^{\frac{2a}{q}-d}g(s)ds}_{L_t^qL_x^\infty}\les& \normo{\int \norm{e^{i(t-s)D^a}D^{\frac{2a}{q}-d}g(s)}_{L_x^\infty}ds}_{L_t^q}\\
\les&\normo{\int |t-s|^{-2/q}\norm{g(s)}_{L_x^1}ds}_{L_t^q}
\les\norm{g}_{L_t^{q'}L_x^1}.
\end{align*}
Therefore we complete the proof.
\end{proof}

\begin{lem} Let $a>0, d\geq 1$ and $0<q<\infty$. Then
\begin{align*}
\norm{D^{\frac{2a}{q}-d}e^{itD^a}\phi}_{L_x^\infty}\leq C |t|^{-\frac{2}{q}}\norm{\phi}_{L_x^1}
\end{align*}
holds if assuming either of the following conditions:

\begin{itemize}
\item $a=1$, $\frac{2}{q}< \frac{d-1}{2}$.

\item $a\neq 1$, $\frac{2}{q}\leq \frac{d}{2}$.
\end{itemize}
\end{lem}
\begin{proof}
By Theorem 1 in \cite{GPW} we get for $0\leq \theta\leq \frac{d-d_a+2}{2}$
\begin{align*}
\norm{e^{itD^a}P_j\phi}_{L_x^\infty}\leq C |t|^{-\theta}2^{j(d-a\theta)}\norm{\phi}_{L_x^1}.
\end{align*}
If $a=1$ and $\frac{2}{q}< \frac{d-1}{2}$, or if $a\neq 1$
and $\frac{2}{q}< \frac{d}{2}$, then we get
\begin{align*}
\norm{D^{\frac{2a}{q}-d}e^{itD^a}\phi}_{L_x^\infty}\leq& C \sum_j \inf_\theta (2^{ja(\frac{2}{q}-\theta)}|t|^{-\theta}\norm{\phi}_{L_x^1})\\
\les& \sum_{2^j\leq |t|^{-1/a}} 2^{2ja/q}\norm{\phi}_{L_x^1}+\sum_{2^j\geq |t|^{-1/a}}2^{ja(\frac{2}{q}-\frac{d-d_a+2}{2})}|t|^{-\frac{d-d_a+2}{2}} \norm{\phi}_{L_x^1}\\
\les& |t|^{-2/q} \norm{\phi}_{L_x^1}.
\end{align*}
Thus it remains to show the case: $a\neq 1, q=\frac{4}{d}$.

By the Young inequality it suffices to show
\begin{align*}
\left |\int e^{it|\xi|^a}e^{ix\xi}\chi_{\leq k}(\xi)|\xi|^{\frac{2a}{q}-d}d\xi\right|\leq C |t|^{-\frac{2}{q}}, \quad \forall k\in \N, \, x\in \R^d.
\end{align*}
Without loss of generality, we may assume $t>0$. By a change of variable $\xi=t^{-1/a}\eta$, it suffices to show 
\begin{align*}
\left |\sum_{j\leq k}\int e^{i|\xi|^a}e^{ix\xi}\chi_{j}(\xi)|\xi|^{\frac{2a}{q}-d}d\xi\right|\leq C, \quad \forall k\in \N, \, x\in \R^d.
\end{align*}
Fix $x\in \R^d$, $k\in \N$. Denote $I_j(x)=\int e^{i|\xi|^a}e^{ix\xi}\chi_{j}(\xi)|\xi|^{\frac{2a}{q}-d}d\xi$. By the Fourier-Bessel formula (see \cite{Stein1}), we have
\begin{align*}
I_{j}(x)=&\int e^{ir^a}\chi_{j}(r)r^{\frac{2a}{q}-1}(r|x|)^{-\frac{d-2}{2}}J_{\frac{d-2}{2}}(r|x|)dr, \quad d\geq 2,\\
I_{j}(x)=&\int e^{i|\xi|^a}e^{ix\xi}\chi_{j}(\xi)|\xi|^{\frac{2a}{q}-d}d\xi, \quad d=1.
\end{align*}
Here $J_\nu(r)$ is the Bessel function defined by
\begin{align}\label{eq:Bessel}
J_\nu(r)=\frac{(r/2)^\nu}{\Gamma(\nu+1/2)\pi^{1/2}}
\int_{-1}^1e^{ir\theta}(1-\theta^2)^{\nu-1/2}d\theta, \ \ \nu>-1/2.
\end{align}

{\bf Case 1:} $2^{j}\les |x|^{-1}$.

First we have the trivial bound $|I_j(x)|\les 2^{\frac{2aj}{q}}$. On the other hand, when $d\geq 2$, by the fact that
\[\aabs{\frac{d^k}{dr^k}(r^{-\nu} J_\nu(r))}\les 1\] and using integration by part  $n$ times we get 
\begin{align*}
|I_j(x)| &= \aabs{\int [\{(iar^{a-1})^{-1}\partial_r\}^n e^{ir^a}]\cdot \chi_{j}(r)r^{\frac{2a}{q}-1}(r|x|)^{-\frac{d-2}{2}}J_{\frac{d-2}{2}}(r|x|)dr}\\
 &\les 2^{\frac{2aj}{q}}2^{-jna}
\end{align*}
for any $n\in \N$. Then we have
$\sum_{2^{j}\les |x|^{-1}}\abs{I_j(x)}\les \sum_{2^{j}\les |x|^{-1}}\min(2^{\frac{2aj}{q}}2^{-jna}, 2^{\frac{2aj}{q}})\les 1$.
Similarly, the same holds for $d=1$. 

{\bf Case 2:} $2^{j}\gg |x|^{-1}$.

Using the fact $r^{-\frac{d-2}{2}}J_{\frac{d-2}{2}}(r)=c_d \Re(e^{ir}h(r))$ where $h$ satisfies $|\partial_r^m h|\les (1+r)^{-\frac{d-1}{2}-m}$  (see Section 1.4, Chapter VIII of \cite{Stein2}) , it suffices to show 
\[\sum_{|x|^{-1}\ll 2^j\leq 2^k}\abs{\tilde I_j(x)}\les 1\] with
\begin{align*}
\tilde I_{j}(x):=&\int_\R e^{i(r^a-r|x|)}\chi_{j}(r)|r|^{\frac{2a}{q}-1}h(r|x|)dr, \quad d\geq 2,\\
\tilde I_{j}(x):=&\int_\R e^{i|\xi|^a}e^{ix\xi}\chi_{j}(\xi)|\xi|^{\frac{2a}{q}-1}d\xi, \quad d=1.
\end{align*}
Hence the 1D case is the same as the higher dimensions with $h(r):\equiv 1$. 

$\bullet$ If $2^{j(a-1)}\sim |x|\gg 2^{-j}$, then $2^j|x|\sim 2^{ja}$ and by the van der Corput lemma (see \cite{Stein2}) we get 
\begin{align*}
\abs{\tilde I_j(x)}\les& 2^{-\frac{j(a-2)}{2}}2^{j(\frac{2a}{q}-1)}(2^j|x|)^{-\frac{d-1}{2}}\les 2^{ja(-\frac{d}{2}+\frac{2}{q})} \le 1.
\end{align*}

$\bullet$ If $2^{j(a-1)}\ll |x|$, integrating by parts $n$ times, we have for any $n\in \N$
\EQ{
\abs{\tilde I_j(x)}
&= \aabs{\int [\{-i(ar^{a-1}-|x|)^{-1}\partial_r\}^n e^{i(r^a-r|x|)}]\cdot\chi_{j}(r)|r|^{\frac{2a}{q}-1}h(r|x|)dr}\\
&\les  (2^j|x|)^{-\frac{d-1}{2}-n}2^{\frac{jad}{2}}.
}

$\bullet$ If $2^{j(a-1)}\gg |x|$, then $2^j\ge 2^{j(1-a)}$ and hence $j\geq 0$. Integrating by parts $n$ times, we have for any $n\in \N$
\EQ{
\abs{\tilde I_j(x)}
 &= \aabs{\int [\{-i(ar^{a-1}-|x|)^{-1}\partial_r\}^n e^{i(r^a-r|x|)}]\cdot\chi_{j}(r)|r|^{\frac{2a}{q}-1}h(r|x|)dr}\\
 &\les (2^j|x|)^{-\frac{d-1}{2}}2^{-jan}2^{j\frac{2a}{q}} \le 2^{-jan}2^{j\frac{2a}{q}}.
}

Therefore, we get
\EQ{
\sum_{|x|^{-1}\ll 2^j\leq 2^k}\abs{\tilde I_j(x)}\les& \sum_{2^{j(a-1)}\sim |x|} 1+\sum_{\min(2^j,2^{(1-a)j})\gg |x|^{-1}}(2^j|x|)^{-\frac{d-1}{2}-n}2^{\frac{jad}{2}}\\
&+\sum_{2^j\gg |x|^{-1}\gg 2^{j(1-a)}} 2^{-jan}2^{\frac{2ja}{q}}\les 1.
}
We complete the proof of the lemma.
\end{proof}

The failure of the estimate \eqref{eq:Strsec2} for $(a,q,d)=(1,4,2)$ was shown in \cite{FangWang}. In the rest of this section we prove Theorem \ref{thm1} by using $a$-stable L\'evy processes\footnote{The authors would like to thank Kais Hamza for the discussion on the L\'evy processes.}. First we collect some properties of these processes.  For $a\in (0,2]$, let 
\begin{align*}
f_a(x)=(2\pi)^{-d}\int_{\R^d}e^{ix\xi}e^{-|\xi|^a}d\xi.
\end{align*}
Then $f_a$ is a smooth strictly positive radial function on $\R^d$ satisfying $\int_{\R^d} f_a(x)dx=1$. In particular, we have
\EQn{ \label{kernel formula}
 \pt f_1(x)=C_1(1+|x|^2)^{-\frac{d+1}{2}}, \quad f_2(x)=C_2 e^{-\frac{|x|^2}{4}}, 
 \pr 0<a<2 \implies f_a(x)\sim (1+|x|)^{-(d+a)},}
see \cite{Levy1}. 
It is well-known that Random variables with distributions given by the density $f_a$ ($0<a\leq 2$) are stable. For $t>0$, let $f_a(t,x)=t^{-d/a}f_a(t^{-1/a}x)$. 

Let $\{Y(t):t\geq 0\}$ be the independent symmetric $a$-stable L\'evy process in $\R^d$ with $Y(0)=0$, that is, a process with the stationary independent increments, and the increment $Y_t-Y_s$ has a distribution given by the density $f_a(|t-s|,x)$. Let $\{\wt Y(t):t\geq 0\}$ be another independent copy of $Y(t)$.  The existence of these processes and their properties were well-understood \cite{Levy2, Levy1}. We construct a process $X_t:=X(t)$ on the whole line $\R$ by defining
\begin{align*}
X(t)=
\begin{cases}
Y(t), \quad t\geq 0;\\
\wt Y(-t), \quad t<0.
\end{cases}
\end{align*}
By this construction we know $X_t-X_s$ has a distribution given by the density $f_a(|t-s|,x)$ for $t\neq s$, and hence
\begin{align}\label{eq:char}
\E e^{i\eta (X_t-X_s)}=\int_{\R^d}e^{i\eta x}f_a(|t-s|,x)dx=e^{-|t-s||\eta|^a}
\end{align}
where $\E$ is the expectation.

Now we prove Theorem \ref{thm1}. The estimate \eqref{eq:crit} is equivalent to
\begin{align*}
\norm{e^{itD^a}\jb{D}^{-\frac{d-a}{2}}f}_{L_t^2L^\infty_x}\leq C \norm{f}_{L^2}.
\end{align*}
By $TT^*$ method, we see it is further equivalent to
\begin{align*}
\normo{\int e^{i(t-s)D^a}[\jb{D}^{a-d}g(s,\cdot)](x)ds}_{L_t^2L^\infty_x}\leq C^2\norm{g}_{L_t^2L_x^1}.
\end{align*}
Assume $g(t,x)=\alpha(t)h(x-X(t))$, where $\alpha\in L^2(\R)$, and $h$ is a Schwartz function. Then the above inequality implies
\begin{align*}
\normo{\int e^{i(t-s)D^a}[\jb{D}^{a-d}g(s,\cdot)](X(t))ds}_{L_t^2}\leq C^2\norm{\alpha}_{L_t^2}\norm{h}_{L_x^1}
\end{align*}
which further implies
\begin{align*}
\normo{\int K(t,s)\alpha(s)ds}_{L_t^2}\leq C^2\norm{\alpha}_{L_t^2}\norm{h}_{L_x^1}
\end{align*}
where
\begin{align*}
K(t,s)=\int_{\R^d}e^{i(t-s)|\xi|^a}(1+|\xi|^2)^{\frac{a-d}{2}}e^{i[X(t)-X(s)]\xi}\hat h(\xi)d\xi.
\end{align*}
By Minkowski's inequality the above inequality implies
\begin{align*}
\normo{\int \E K(t,s)\alpha(s)ds}_{L_t^2}\leq C^2\norm{\alpha}_{L_t^2}\norm{h}_{L_x^1}
\end{align*}
where
\begin{align*}
\E K(t,s)
&=\E \int_{\R^d}e^{i(t-s)|\xi|^a}(1+|\xi|^2)^{\frac{a-d}{2}}e^{i(X_t-X_s) \xi}\hat h(\xi)d\xi\\
&=\int_{\R^d}e^{i(t-s)|\xi|^a}(1+|\xi|^2)^{\frac{a-d}{2}} e^{-{|t-s|\cdot |\xi|^a}}\hat h(\xi)d\xi\\
&=:K_h(t-s).
\end{align*}
In the second equality above we used \eqref{eq:char}.
For a fixed function $h$, the operator given by the kernel $K_h$ is a convolution operator, and thus its $L^2$ boundedness implies
\begin{align}\label{eq:Kbound}
\norm{\widehat{K_h}(\tau)}_{L_\tau^\infty} \leq C \norm{h}_{L^1},
\end{align}
for any Schwartz function $h$.

To prove Theorem \ref{thm1}, it suffices to disprove \eqref{eq:Kbound}. By direct calculation we get
\begin{align*}
\widehat{K_h}(\tau)=C\int_{\R^d}\frac{|\xi|^a(1+|\xi|^2)^{\frac{a-d}{2}}}{(\tau-|\xi|^a)^2+|\xi|^{2a}}\hat h(\xi)d\xi.
\end{align*}
Thus $\widehat{K_h}(0)=\int_{\R^d}|\xi|^{-a}(1+|\xi|^2)^{\frac{a-d}{2}}\hat h(\xi)d\xi$. The integrand for large $\xi$ is essentially $|\xi|^{-d}\hat h(\xi)$ which makes the integral a logarithmic infinity. For example, one can take $h$ to be an approximating sequence of $\delta(x)$. Then clearly, \eqref{eq:Kbound} fails. 

\begin{rem}
The above proof only use the fact that $X(t)$ satisfies \eqref{eq:char}. For the wave equation (namely when $a=1$), we can take $X(t)=tZ$, where $Z$ is a random variable with characteristic function $e^{-|\eta|}$. 
\end{rem}

\section{Strichartz estimates with angular smoothing effect}
In this section, we prove Theorem \ref{thm2}.  It is
equivalent to show
\begin{align}\label{eq:Stri2}
\norm{\Lambda_\om^s (T_af)}_{L_t^2 L_{\rho}^\infty L_\om^2}\les \norm{f}_{L_x^2},
\end{align}
where
\[T_af(t,x)=\int_{\R^d}e^{i(x\xi+t|\xi|^a)}|\xi|^{\frac{a-d}{2}}f(\xi)d\xi.\]
Now we apply the spherical-radius decomposition to $f$ (see \cite{Stein1})
\[f(\xi)=f(\rho \om)=\sum_{k\geq 0}\sum_{1\leq l\leq n(k)}a_k^l(\rho)Y_k^l(\om), \quad \rho=|\xi|,\ \om=\frac{\xi}{|\xi|}\in \cir^{d-1},\]
where $k\geq
0,1\leq l\leq n(k)$, $n(k)=C_{d+k-1}^k-C_{d+k-3}^{k-2}$, $\{Y_k^l\}$ is the standard orthonormal basis (spherical harmonics of degree $k$) in $L^2(\cir^{d-1})$. We have (see \cite{Stein1})
\[T_af(t,x)=\sum_{k,l}c_{d,k}T_a^\nu (\rho^{\frac{d-1}{2}}a_k^l)(t,|x|)Y_k^l(x/|x|),\]
where $c_{d,k}=(2\pi)^{d/2}i^{-k}$, $\nu=\nu(k)=\frac{d-2+2k}{2}$, and
\[T_a^\nu(h)(t,r)=r^{-\frac{d-2}{2}}\int_0^\infty e^{it\rho^a}J_\nu(r\rho)\rho^{\frac{-d+1+a}{2}}h(\rho)d\rho.\]
Here $J_\nu(r)$ is the Bessel function given by \eqref{eq:Bessel}.
Thus \eqref{eq:Stri2} is equivalent to 
\begin{align}\label{eq:Stri3}
\norm{(1+|k|)^sT_a^\nu (a_k^l)}_{L_t^2L_{r}^\infty l_{k,l}^2}\les
\norm{\{a_k^l(\rho)\}}_{L_\rho^2l_{k,l}^2}.
\end{align}
To prove \eqref{eq:Stri3}, it is equivalent to show
\begin{align}\label{eq:goal}
\norm{T_a^\nu (h)}_{L_t^2 L_{r}^\infty}\leq C (1+\nu)^{-s}\norm{h}_{L^2},
\end{align}
with constant $C$ independent of $\nu$.

Decompose $T_a^\nu (h)=\sum_{j\in \Z}T_{a,j}^\nu (h)$, where 
\[T_{a,j}^\nu(h)(t,r)=\chi_j(r)r^{-\frac{d-2}{2}}\int_0^\infty e^{it\rho^a}J_\nu(r\rho)\rho^{\frac{-d+1+a}{2}}h(\rho)d\rho.\]
Then obviously
\begin{align*}
\norm{T_a^\nu (h)}_{L_t^2 L_{r}^\infty}\les&\normo{T_{a,j}^\nu (h)}_{L_t^2l^\infty_j L_r^\infty }.
\end{align*}
We decompose further $T_{a,j}^\nu(h)=T_{a,j,*}^\nu(h)+\sum_{j'\geq -j-4}T_{a,j,j'}^\nu(h)$, where
\begin{align*}
T_{a,j,*}^\nu(h)(t,r)=&\chi_j(r)r^{-\frac{d-2}{2}}\int_0^\infty e^{it\rho^a}J_\nu(r\rho)\rho^{\frac{-d+1+a}{2}}\chi_{\leq -j-5}(\rho)h(\rho)d\rho,\\
T_{a,j,j'}^\nu(h)(t,r)=&\chi_j(r)r^{-\frac{d-2}{2}}\int_0^\infty e^{it\rho^a}J_\nu(r\rho)\rho^{\frac{-d+1+a}{2}}\chi_{j'}(\rho)h(\rho)d\rho.
\end{align*}
Then we have
\[\norm{T_{a,j}^\nu(h)}_{L_t^2l^\infty_j L_r^\infty }\leq \norm{T_{a,j,*}^\nu(h)}_{L_t^2l^\infty_j L_r^\infty }+\normo{\sum_{j'\geq -j-4}\norm{T_{a,j,j'}^\nu(h)}_{L_t^2 L_r^\infty }}_{l^2_j }.\]

First we control $T_{a,j,*}^\nu$. We will use the vanishing properties of $J_\nu(r)$ near $r=0$.  Using the formula \eqref{eq:Bessel} and Taylor's expansion for $e^{irt}$ we get
\begin{align*}
T_{a,j,*}^\nu (h)=\chi_j(r)r^{-\frac{d-2}{2}}\int_0^\infty \frac{e^{it\rho^a}(r\rho/2)^\nu}{\Gamma(\nu+1/2)\pi^{1/2}}&\int_{-1}^1 \sum_{n=0}^\infty \frac{(i\theta r\rho)^n}{n!}(1-\theta^2)^{\nu-1/2}d\theta\\
&\times\rho^{\frac{-d+1+a}{2}}\chi_{\leq -j-5}(\rho)h(\rho)d\rho.
\end{align*}
Then we have
\begin{align*}
|T_{a,j,*}^\nu (h)|\les&\sum_{n=0}^\infty\chi_j(r)2^{-j\frac{d-2}{2}}2^{j\nu}2^{jn}\bigg|\int_0^\infty \frac{e^{it\rho^a}(\rho/2)^\nu}{\Gamma(\nu+1/2)\pi^{1/2}}\\
&\int_{-1}^1
\frac{(i\theta \rho)^n}{n!}(1-\theta^2)^{\nu-1/2}d\theta\rho^{\frac{-d+1+a}{2}}\chi_{\leq -j-5}(\rho)h(\rho)d\theta d\rho\bigg|.
\end{align*}
Making a change of variable $\eta=\rho^a$, and then using Plancherel's equality, we get
\begin{align*}
&\norm{T_{a,j,*}^\nu (h)}_{L_t^2 l_j^\infty L_r^\infty }\\
\les&\frac{1}{\Gamma(\nu+1/2)}\sum_{n=0}^\infty \frac{C^n}{n!}\normo{\sup_j\left|\int_0^\infty e^{it\rho^a}(2^j\rho)^{\nu+n-\frac{d-2}{2}}\chi_{\leq -j-5}(\rho)\rho^{\frac{a-1}{2}}h(\rho)d\rho\right|}_{L^2_t}\\
\les& \frac{1}{\Gamma(\nu+1/2)}\sum_{n=0}^\infty \frac{C^n}{n!}\normo{\sup_j\left|\int_0^\infty e^{it\rho}(2^j\rho^{1/a})^{k+n}\chi_{\leq -j-5}(\rho^{1/a})h(\rho^{1/a})\rho^{-\frac{a-1}{2a}}d\rho\right|}_{L^2_t}\\
=:&\frac{1}{\Gamma(\nu+1/2)}\sum_{n=0}^\infty \frac{C^n}{n!}A_n.
\end{align*}
Let $\tilde h(\rho)=h(\rho^{1/a})\rho^{-\frac{a-1}{2a}}1_{[0,\infty)}(\rho)$. Then $\norm{\tilde h}_2\sim \norm{h}_2$.  If $k+n=0$, then 
\begin{align*}
A_n
\les&\normo{\sup_j\left|\int e^{it\rho}\chi_{\leq -j-5}(\rho^{1/a})\tilde h(\rho)d\rho\right|}_{L^2_t}\\
\les& \norm{M(\hat{\tilde h})(t)}_{L_t^2}\les \norm{h}_2,
\end{align*}
where $M$ is the Hardy-Littlewood maximal operator since $\chi_{\leq 1}(\rho^{1/a})$ is a Schwartz function.
If $k+n\geq 1$, then by Plancherel's equality we get
\begin{align*}
A_n
\les&\normo{\int e^{it\rho}(2^j\rho^{1/a})^{k+n}\chi_{\leq -j-5}(\rho^{1/a})\tilde h(\rho)d\rho}_{l_j^2L^2_t}\\
\les&\normo{(2^j\rho^{1/a})^{k+n}\chi_{\leq -j-5}(\rho^{1/a})\tilde h(\rho)}_{l_j^2L^2_\rho}\\
\les&\normo{\sum_{j+j'\leq -5}2^{(j+j')(k+n)}\chi_{j'}(\rho^{1/a})\tilde h(\rho)}_{l_j^2L^2_\rho}\les \norm{\tilde h}_{2}\les \norm{h}_2.
\end{align*}
So we get 
\begin{align*}
\norm{T_{a,j,*}^\nu (h)}_{L_t^2 l_j^\infty L_r^\infty }\les (1+\nu)^{-K} \norm{h}_2, \quad \forall\ K\in \N.
\end{align*}

Next we control $T_{a,j,j'}^\nu$.  We will use the asymptotic behaviour of $J_\nu(r)$ for $r\to \infty$. We have
\begin{align*}
\norm{T_{a,j,j'}^\nu(h)}_{L_t^2L_r^\infty}\les &2^{-j\frac{d-2}{2}}2^{-j'a/2}2^{j'\frac{-d+2+a}{2}}\normb{\chi_{j+j'}(r)\int_0^\infty e^{it\rho^a}J_\nu(r\rho)\\
&\qquad\qquad\qquad\qquad\qquad\times\rho^{\frac{-d+1+a}{2}}\chi_{0}(\rho)h(2^{j'}\rho)2^{j'/2}d\rho}_{L_t^2L_r^\infty}\\
 =&2^{-(j+j')\frac{d-2}{2}}\norm{S^\nu_{j+j'}(h_{j'})}_{L_t^2L_r^\infty},
\end{align*}
where we denote  $h_{j'}:=h(2^{j'}\rho)2^{j'/2}$ and the operator $S^\nu_j$ is defined by 
\[S^\nu_{j}(h):=\chi_{j}(r)\int_0^\infty e^{it\rho^a}J_\nu(r\rho)\rho^{\frac{-d+1+a}{2}}\chi_{0}(\rho)h(\rho)d\rho.\] 
It is the same operator as $S_R^{\nu,a}$ with $R=2^{j}$ that was studied in \cite{Guo}, or the operator $T_{j,k}^{\nu}$ with $\omega(\rho)=\rho^a, k=0$ studied in \cite{GHN}. 

{\bf Case 1:} $\nu\les 2^{j+j'}$.

In this case, we can use the result in Lemma 3.6 in \cite{GHN} for $T_{j,k}^{\nu}$ and then get
\EQ{
\norm{S^\nu_{j+j'}(h)}_{L_t^2L_r^\infty}\les \norm{h}_2,
}
and thus
\[\norm{T_{a,j,j'}^\nu(h)}_{L_t^2L_r^\infty}\les 2^{-(j+j')\frac{d-2}{2}}\norm{\chi_{j'}(\rho)h(\rho)}_2.\]
Therefore, we get
\begin{align*}
\normo{\sum_{j'+j\geq -4}\norm{T_{a,j,j'}^\nu (h)}_{L_t^2L_r^\infty}}_{l_j^2}\les \nu^{-\frac{d-2}{2}+\epsilon}\norm{h}_2.
\end{align*}
This suffices to show Part (1) of Theorem \ref{thm2}.

{\bf Case 2:} $\nu\gg 2^{j+j'}$.

In this case,  we use the Stirling formula for the Gamma function and get
\[\Gamma(\nu+1)\geq C\nu^{1/2}(\nu/e)^\nu\]
from which we get better decay
\begin{align*}
|J_\nu(r)|+|J_\nu'(r)|\leq& C\nu^{-K}, \quad \forall\ K\in \N.
\end{align*}
By the above bound and the Sobolev embedding we get
\begin{align}\label{eq:Sbound}
\norm{S^\nu_{j+j'}(h)}_{L_t^2L_r^\infty}\les \norm{S^\nu_{j+j'}(h)}_{L_t^2L_r^2}+\norm{\partial_r S^\nu_{j+j'}(h)}_{L_t^2L_r^2}\les \nu^{-K} \norm{h}_2,\quad \forall\ K\in \N.
\end{align}
and thus
\[\norm{T_{a,j,j'}^\nu(h)}_{L_t^2L_r^\infty}\les \nu^{-K}2^{-(j+j')\frac{d-2}{2}}\norm{\chi_{j'}(\rho)h(\rho)}_2.\]
Therefore, we get
\begin{align*}
\normo{\sum_{j'+j\geq -4}\norm{T_{a,j,j'}^\nu (h)}_{L_t^2L_r^\infty}}_{l_j^2}\les \nu^{-K}\norm{h}_2, \quad \forall \ K\in \N
\end{align*}
which suffices for our purpose.

When $1<a<\infty$, we can improve Case 1. Applying the results in Lemma 3.10 in \cite{GHN} by taking $\lambda=R^{3/7}$ we get
\begin{align*}
\norm{S^\nu_{j+j'}(h)}_{L_t^2L_r^\infty}\les 2^{-(j+j')/7}  \norm{h}_2.
\end{align*}
Therefore, for the case $1<a<\infty$ we get
\begin{align*}
\normo{\sum_{j'+j\geq -4}\norm{T_{a,j,j'}^\nu (h)}_{L_t^2L_r^\infty}}_{l_j^2}\les \nu^{-\frac{7d-12}{14}+\e}\norm{h}_2.
\end{align*}
This suffices to show Part (2) of Theorem \ref{thm2}.

\begin{rem}
For $0<a\leq 1$ and $d=2$, by the similar arguments we can get the following estimate: if $q>2$
\begin{align}\label{eq:halfwave4infty}
\norm{\Lambda_\om^{\frac{1}{2}-\frac{1}{q}-\epsilon}e^{itD^a}f}_{L_t^q\Lr^\infty_\rho L^2_{\om}(\R\times \R^2)}\les \norm{f}_{\dot H^{\frac{2-a}{2}}}, \quad \forall \epsilon>0.
\end{align}
Indeed, to prove \eqref{eq:halfwave4infty}, we just interpolate the estimates $\norm{T^\nu_{a,j,j'}(h)}_{L_t^2L_r^\infty}\les \norm{h}_{L^2}$ with the following estimate
\begin{align*}
\norm{T^\nu_{a,j,j'}(h)}_{L_t^\infty L_r^\infty}\les 2^{-(j+j')/2}\norm{h}_{L^2}.
\end{align*}
Then we use the $\epsilon$-room to do the summation.
\end{rem}

\section{Inhomogeneous Strichartz estimates}

In this section we prove Theorem \ref{thm3}. These double endpoint inhomogeneous Strichartz estimates have useful applications of controlling potential terms. These estimates can not be deduced directly by Christ-Kiselev lemma from homogeneous estimates.

\begin{proof}[Proof of Theorem \ref{thm3}]
We may assume $q\geq r$, since otherwise we consider the adjoint estimate. Thus we may further assume $q=r$ by Bernstein's inequality. It suffices to prove
\begin{align*}
|T(F,G)|\les \norm{F}_{L_t^2\Lr_\rho^{r'}L_\om^2}\norm{G}_{L_t^2\Lr_\rho^{r'}L_\om^2}
\end{align*}
where
\[T(F,G)=\int_{-\infty}^\infty \int_{-\infty}^t \jb{e^{-isD^a} P_0F(\cdot,s), e^{-itD^a}P_0G(\cdot,t)}_{L_x^2}dsdt.\]
Decompose $T$ dyadically $T(F,G)=\sum_j T_j(F,G)$ where
\[T_j(F,G)=\int_{-\infty}^\infty \int_{t-2^{j+1}<s\leq t-2^j} \jb{e^{-isD^a} P_0F(\cdot,s), e^{-itD^a}P_0G(\cdot,t)}_{L_x^2}dsdt.\]
It suffices to prove
\begin{align*}
\sum_{j\in \Z}|T_j(F,G)|\les \norm{F}_{L_t^2\Lr_\rho^{r'}L_\om^2}\norm{G}_{L_t^2\Lr_\rho^{r'}L_\om^2}.
\end{align*}

Now we consider the following estimates for $T_j(F,G)$:
\begin{align*}
|T_j(F,G)|\les& C(j) \norm{F}_{L_t^2\Lr_\rho^{r'}L_\om^2}\norm{G}_{L_t^2\Lr_\rho^{r'}L_\om^2}.
\end{align*}
We may assume $F,G$ both have compact supports in $t$ on an interval of length $O(2^j)$.
First we have the trivial estimates
\begin{align*}
|T_j(F,G)|\les& \norm{P_0F}_{L_t^1L_x^{2}}\norm{P_0G}_{L_t^1L_x^{2}}\les 2^j \norm{P_0F}_{L_t^2 L_x^2}\norm{P_0G}_{L_t^2 L_x^2}\\
\les&2^j \norm{F}_{L_t^2 L_x^{r'}}\norm{G}_{L_t^2 L_x^{r'}}
 \les 2^j \norm{F}_{L_t^2\Lr_\rho^{r'}L_\om^2}\norm{G}_{L_t^2\Lr_\rho^{r'}L_\om^2}
\end{align*}
This estimate suffices to sum over $j\leq 0$.  It remains to consider $j>0$. By the dispersive estimate we have
\begin{align}\label{eq:decay}
|T_j(F,G)|\les& 2^{-\frac{j(d+2-d_a)}{2}}\norm{F}_{L_t^1L_x^{1}}\norm{G}_{L_t^1L_x^{1}}\les 2^{-\frac{j(d+2-d_a)}{2}}2^j \norm{F}_{L_t^2L_x^{1}}\norm{G}_{L_t^2L_x^{1}}.
\end{align}
Fix $a\in (\frac{4d+2-2d_a}{2d-d_a-1}, r)$. By the Christ-Kiselev lemma we can get for any $\e>0$
\begin{align*}
\normo{\int_{|t-s|\sim 2^j} e^{i(t-s)D^a}P_0 f(s)ds}_{L_t^{2(\e)}\Lr_\rho^a L^2_\om}\les \norm{f}_{L_t^2\Lr_\rho^{a'}L^2_\om}
\end{align*}
where 
\[\frac{1}{2(\e)}=\frac{1}{2}-\e.\]
Then we get
\begin{align}\label{eq:ep}
|T_j(F,G)|\les& 
\normo{\int_{|t-s|\sim 2^j} e^{i(t-s)D^a}P_0 F(s)ds}_{L_t^{2(\e)}\Lr_\rho^a L^2_\om} \norm{G}_{L_t^{2(-\e)}\Lr_\rho^{a'}L^2_\om}\nonumber\\
\les&\norm{F}_{L_t^{2}\Lr_\rho^{a'}L^2_\om}\norm{G}_{L_t^{2(-\e)}\Lr_\rho^{a'}L^2_\om}\les 2^{j\e}\norm{F}_{L_t^{2}\Lr_\rho^{a'}L^2_\om}\norm{G}_{L_t^{2}\Lr_\rho^{a'}L^2_\om}
\end{align}
Interpolating \eqref{eq:decay} and \eqref{eq:ep}, and choosing $\e$ sufficiently small, we get that for some $\theta>0$
\begin{align*}
|T_j(F,G)|\les& 2^{-j\theta}\norm{F}_{L_t^{2}\Lr_\rho^{r'}L^2_\om}\norm{G}_{L_t^{2}\Lr_\rho^{r'}L^2_\om}.
\end{align*}
Thus we can sum over $j\geq 0$ and hence complete the proof.
\end{proof}

In the above proof, we see the dispersive estimates of rate $|t|^{-\theta}$ with some $\theta>1$ is crucial. If the decay rate $\theta=1$, the estimates may fail.  Indeed, in \cite{Tao} Tao showed that 
\begin{align*}
\normo{\int_0^t e^{-i(t-s)\Delta} f(s)ds}_{L_t^2L_x^\infty(\R\times \R^2)}\les \norm{f}_{L_t^2L_x^1}
\end{align*}
fails even when $f$ is radial in $x$.
Finally we observe its extensions to general dimensions. 
Condider a general form of Strichartz estimate 
\EQ{
\normo{\int_{-\I}^t e^{i(t-s)D^a}f(s)ds}_{L^p_t X^*} \lec \|f\|_{L^{q'}_t Y},}
where $X,Y$ are two Banach spaces embedded into $\cS'(\R^d)$. 
By duality, it is equivalent to 
\EQ{
 \iint_{s<t}(e^{i(t-s)D^a}f(s)|g(t)) dsdt |\lec \|f\|_{L^{p'}_t X}\|g\|_{L^{q'}_t Y}.}
Restriction to the functions with separated variables yields 
\EQ{
 \pt K_{\fy,\psi}(t):=(e^{itD^a}\fy|\psi) 
 \pr\implies \iint_{s<t}K_{\fy,\psi}(t-s)f(s)g(t)dsdt |\lec \|f\|_{L^{p'}_t}\|g\|_{L^{q'}_t}\|\fy\|_{X}\|\psi\|_{Y}.}
A simple case of $K(t)$ is when, with a parameter $\s>0$, 
\EQ{
 \pt \hat\fy=|\x|^\al \s^\ga e^{-\s|\x|^a/2}, \pq \hat\psi=|\x|^\be \s^{-\ga} e^{-\s|\x|^a/2},
 \pq \al+\be=a-d.}
We can explicitly compute 
\EQ{
 c_dK(t)=\int_0^\I e^{itr^a-\s r^a}ar^{a-1}dr
 =\int_0^\I e^{its-\s s}ds=\frac{1}{\s-it}
 =\frac{\s+it}{\s^2+t^2}.}
If we have a uniform bound of the above estimate for such $\fy,\psi$ with $\s\to0+$, then the limit after taking the imaginary part is 
\EQ{
 \iint_{s<t}\frac{1}{t-s}f(s)g(t)dsdt,}
which is clearly divergent, for any $p,q\in[1,\I]$. 
Thus we have obtained the following criterion for the Strichartz estimate, not necessarily at the endpoint. 

\begin{prop}
Let $d\in\N$, $a>0$, $\al,\be\in(-d,\I)$ such that $\al+\be=a-d$. 
Let $X,Y$ be two Banach spaces of functions embedded into $\cS'(\R^d)$. 
Suppose that the pair of functions with a parameter $\s>0$ 
\EQ{
 \F^{-1}(|\x|^\al \s^\ga e^{-\s|\x|^a/2}),
 \pq \F^{-1}(|\x|^\be \s^{-\ga} e^{-\s|\x|^a/2}) }
are bounded as $\s\to+0$ respectively in $X$ and in $Y$. 
Then for any $p,q\in[1,\I]$, the following estimate is false. 
\EQ{
 \normo{\int_{-\I}^t e^{i(t-s)D^a}f(s)ds}_{L^p_t X^*} \lec \|f\|_{L^{q'}_t Y}.}
\end{prop}

The most typical scaling (including Tao's case $a=2$) is 
\EQ{
 \al=\ga=0,\pq \be=a-d, \pq X=\Lr^1_r L^\I_\te,\pq Y=D^{a-d}X.}
Then we have, using \eqref{kernel formula},  
\EQ{
 &\F^{-1}(|\x|^\al \s^\ga e^{-\s|\x|^a/2}) =D^{d-a}\F^{-1}(|\x|^\be \s^{-\ga} e^{-\s|\x|^a/2})
 =\F^{-1}e^{-\s|\x|^a/2} \in X,}
so that we can apply the above proposition.  
Explicitly, the following inequality
\EQ{
\normo{\int_0^t e^{i(t-s)D^a} f(s)ds}_{L^p_t \Lr^\I_r L^1_\omega} \lec \|D^{d-a}f\|_{L^{q'}_t \Lr^{1}_r L^\I_\omega}}
fails, even when $f$ is radial in $x$. Note however we can not replace the norms by 
\EQ{
 \Lr^\I_r \to \dot B^0_{\I,p} \text{ or } \Lr^1_r \to \dot B^0_{1,p}}
for any $p>1$, since $\dot B^0_{1,q}$ with $q<\I$ does not include the Gauss functions. 
In fact, $f\in\dot B^0_{1,q}$ with $q<\I$ implies that
\EQ{
 \|\hat f\|_{L^\I(|\x|\sim N)} \lec \|f_N\|_{L^1} \to 0 \pq(N\to+0).}

\appendix

\section{Cubic fractional Schr\"odinger equations}

In the appendix, we consider the Cauchy problem to the fractional Schr\"odinger equation
\EQn{\label{eq:fcNLS}
iu_t+D^au=|u|^2u,\quad u(0,x)=\phi.
}
By scaling invariance: for $\lambda>0$,
\[u(t,x)\to u_\lambda=\lambda^{a/2}u(\lambda^a t, \lambda x), \quad \phi\to \lambda^{a/2}\phi(\lambda x).\]
The critical Sobolev space in the sense of scaling is $\dot H^{s_c}$ with $s_c=\frac{d-a}{2}$ since $\norm{\lambda^{a/2}\phi(\lambda x)}_{\dot H^{s_c}}=\norm{\phi}_{\dot H^{s_c}}$. We prove the following results.
\begin{thm}\label{thm4}
Assume $d=2$, $0<a<2$ and $a\ne 1$, $\phi\in \dot H^{\frac{2-a}{2},1}$. Then the Cauchy problem \eqref{eq:fcNLS} is locally well-posed. Moreover, if $\norm{\phi}_{\dot H^{\frac{2-a}{2},1}}$ is sufficiently small, then we have global well-posedness and scattering. 
\end{thm}

The space $\dot H^{\frac{2-a}{2},1}$ is the Sobolev space $\dot H^{\frac{2-a}{2}}$ with additional one order angular regularity and $\norm{\phi}_{\dot H^{\frac{2-a}{2},1}}=\norm{\phi}_{\dot H^{\frac{2-a}{2}}}+\norm{\partial_\theta \phi}_{\dot H^{\frac{2-a}{2}}}$. Unfortunately, we can't cover the interesting case $a=1$ which is the energy-critical half wave equation since the crucial $L_t^2L_x^\infty$ estimate fails at $d=2$ in the radial case (see \cite{GuoWang}). 

We prove Theorem \ref{thm4} by the standard iteration arguments using the $L_x^\infty$ type estimate as in \cite{HSire}. We only consider the radial case. By Duhamel's principle, we get 
\EQ{
u=\Phi_\phi(u):=e^{itD^a} \phi-i\int_0^t e^{i(t-s)D^a}|u(s)|^2u(s)ds.
}
For an interval $I$, define the resolution space $X_I$:
\EQ{
X_I=\{u: \mbox{radial}, \norm{u}_{L_I^2L_x^\infty}\leq \eta, \norm{u}_{L_I^\infty \dot H_x^{\frac{2-a}{2}}}\leq M\}
}
endowed with a distance $d(u,v)=\norm{u-v}_{L_I^2L_x^\infty\cap L_I^\infty \dot H_x^{\frac{2-a}{2}}}$, where $\eta, M$ will be determined later to make $\Phi_\phi: (X_I,d)\to (X_I,d)$ a contraction mapping. Then by fractional Leibniz rule we get
\EQ{
\norm{\Phi_\phi(u)}_{L_I^2L_x^\infty}\les& \norm{e^{itD^a} \phi}_{L_I^2L_x^\infty}+C\norm{|u|^2u}_{L_I^1\dot H_x^{\frac{2-a}{2}}}\\
\les&\norm{e^{itD^a} \phi}_{L_I^2L_x^\infty}+\norm{u}_{L_I^2L_x^\infty}^2\norm{u}_{L_I^\infty \dot H_x^{\frac{2-a}{2}}}.
}
Thus we can choose suitable $\eta, M$ to close the iteration arguments.

\end{document}